\documentclass[11pt,leqno]{amsart}
\usepackage{amsthm,amsfonts,amssymb,amsmath,oldgerm}
\numberwithin{equation}{section}
\usepackage{fullpage}
\usepackage{amsmath}
\usepackage{amsfonts}
\usepackage{amssymb}
\usepackage{setspace}
\usepackage{stmaryrd}
\usepackage{mathrsfs}
\usepackage{fancyhdr}
\usepackage{graphicx}
\usepackage{enumerate}
\usepackage{bm}
\usepackage{multirow}
\usepackage{psfrag}
\usepackage[font=scriptsize]{caption}
\usepackage[font=scriptsize]{subcaption}
\usepackage{listings}
\usepackage{tikz}
\usetikzlibrary{matrix} 
\usepackage[framemethod=TikZ]{mdframed}
\mdfdefinestyle{MyFrame}{backgroundcolor=gray!20!white}
\usepackage[makeroom]{cancel}

\usepackage[top=30mm,bottom=30mm,left=22mm,right=22mm,a4paper]{geometry}

\usepackage{hyperref,bookmark}

\usepackage[normalem]{ulem}
\definecolor{violet}{rgb}{0.580,0.,0.827}

\usepackage[textwidth= \marginparwidth,textsize=tiny]{todonotes}

\hypersetup{
    colorlinks,          
    citecolor=blue,        
    filecolor=blue,      
    urlcolor=blue,           
    linkcolor=blue,
}

\newcommand\dD{\mathrm{d}}

\newcommand{\J}{{\mathcal J}}

\newcommand\br{\begin{remark}}
\newcommand\er{\end{remark}}
\newcommand\bp{\begin{pmatrix}}
\newcommand\ep{\end{pmatrix}}
\newcommand{\be}{\begin{equation}}
\newcommand{\ee}{\end{equation}}
\newcommand\ba{\begin{equation}\begin{aligned}}
\newcommand\ea{\end{aligned}\end{equation}}
\newcommand\ds{\displaystyle}

\newcommand{\beg}{\begin{example}}
\newcommand{\eeg}{\end{exaplem}}
\newcommand{\bpr}{\begin{proposition}}
\newcommand{\epr}{\end{proposition}}
\newcommand{\bt}{\begin{theorem}}
\newcommand{\et}{\end{theorem}}
\newcommand{\bc}{\begin{corollary}}
\newcommand{\ec}{\end{corollary}}
\newcommand{\bl}{\begin{lemma}}
\newcommand{\el}{\end{lemma}}
\newcommand{\bd}{\begin{definition}}
\newcommand{\ed}{\end{definition}}
\newcommand{\brs}{\begin{remarks}}
\newcommand{\ers}{\end{remarks}}

\newtheorem{theorem}{Theorem}[section]
\newtheorem{proposition}[theorem]{Proposition}
\newtheorem{corollary}[theorem]{Corollary}
\newtheorem{lemma}[theorem]{Lemma}
\newtheorem{remark}[theorem]{Remark}
\newtheorem{definition}[theorem]{Definition}

\newtheorem{example}[theorem]{Example}

\newcommand\R{\mathbf R}

\newcommand\bj{{j}}
\newcommand\bk{{\bm k}}
\newcommand\bll{{\bm l}}
\newcommand\bmm{{\bm m}}
\newcommand\bx{{\bm x}}
\newcommand{\bz}{\bm z}

\newcommand{\ber}{\bm e}
\newcommand{\bu}{\bm u}
\newcommand\bv{{\bm v}}
\newcommand\Div{{\rm div}}

\newcommand\bA{{\mathbf A}}

\newcommand\bI{{\mathbf I}}

\newcommand\bP{{\mathbf P}}

\newcommand\cB{{\mathcal B}}
\newcommand\cC{{\mathcal C}}

\newcommand\cH{{\mathcal H}}
\newcommand\cI{{\mathcal I}}

\newcommand\cL{{\mathcal L}}
\newcommand\cM{{\mathcal M}}

\newcommand\cP{{\mathcal P}}

\newcommand\cR{{\mathcal R}}

\newcommand\Supp{\textrm{Supp}}

\newcommand{\proj}{\cI}
\usetikzlibrary{fit}
\usetikzlibrary{arrows}
\usepackage{tikz}

\title[A spectral collocation method for the Landau equation]{A spectral collocation method for the Landau equation in plasma
  physics}  

\author{Francis Filbet}
\address{Universit\'e de Toulouse III \& IUF, 
UMR CNRS 5219, Institut de Math\'ematiques de Toulouse,
118 route de Narbonne,
F-31062 Toulouse cedex, France}
\email{francis.filbet@math.univ-toulouse.fr }
\thanks{The author is supported by the EUROfusion Consortium and has received funding
from the Euratom research and training programme 2014--2018 under the grant
agreement No 633053. The views and opinions expressed herein do not
necessarily reflect those of the European Commission.}

\begin{document}

\begin{abstract}
  In this paper we present a spectral collocation method for the fast evaluation
of the Landau collision operator for plasma physics, which
allows us to obtain spectrally accurate numerical solutions. The method is
inspired by the seminal work \cite{PRT2}, but it is specifically designed for Coulombian interactions, taking into account
the particular structure of the operator. It allows us to reduce the
number of discrete convolutions to provide an approximation
of the Landau operator.  Then,  we show that the method preserves the total
mass whereas momentum and energy are approximated with spectral
accuracy. Numerical results for the Landau equation in three dimensions in velocity
space are presented to illustrate the efficiency of the present approach.
\end{abstract}

\date{\today}
\maketitle

\vspace{0.1cm}
\\

\vspace{0.1cm}


\tableofcontents


\section{Introduction}
\label{sec:introduction}
\setcounter{equation}{0}

The Landau or Fokker-Planck-Landau equation is a common kinetic model
applied to collisional  plasmas. More precisely, this mathematical model describes binary collisions between charged particles with
long-range Coulomb interaction and is represented by a nonlinear
partial integro-differential equation written as, 
\be
\label{eq:0}
\frac{\partial f}{\partial t} + \bv.\nabla_\bx f = \cC(f,f),
\ee
where the collision operator $\cC(f,f)$ is given by 
\be
\cC(f,f) = \nabla_\bv\cdot\int_{\R^3}\bA(\bv-\bv^*)\biggr[
  \nabla_\bv f(\bv)\,f(\bv^*) \,-\, \nabla_{\bv^*}f(\bv^*)\,f(\bv)\biggr] \dD
  \bv^*, 
\label{eq:1}
  \ee
with the matrix $\bA(\bz)$,
\begin{equation}
\label{def:A}
\bA(\bz) = \,|\bz|^{\gamma +2}\,\left( \bI - \frac{\bz\otimes \bz}{|\bz|^2}\right)\,,
\end{equation}
where $\bI$ represents the identity matrix in dimension three. The function $f$ represents the density of a gas in 
phase space at positions $\bx$ and velocities $\bv$, it is nonnegative
and integrable together with its moments in velocity
up to the second order $f\in L^1((1+|\bv|^2)\dD \bx\dD\bv)$. Actually
it is worth to mention that different values of $\gamma$ lead to usual
classification in hard potentials $\gamma > 0$, Maxwellian molecules
$\gamma = 0$, or soft potentials  $\gamma < 0$. However, this latter
case is of primary importance in plasma physics since it 
involves the Coulombian case when $\gamma = -3$. Furthermore, the algebraic structure of the operator is similar to the Boltzmann
one for rarefied gas dynamics, this leads to physical properties such as the conservation of mass, impulsion, and energy
\begin{eqnarray*}
\int_{\R^3} \cC(f,f)(\bv)\left(
\begin{array}{c}
1 \\          
\bv \\
|\bv|^2 
\end{array}
\right)
\dD \bv \,=\, 0,
\end{eqnarray*}
and the decay of the kinetic entropy $\cH[f](t)$,
\be
\label{eq:H}
\frac{\dD\,\cH[f]}{dt} \,=\, \frac{\dD}{\dD t}\int_{\R^3} f(\bv)\,\ln
(f(\bv))\,\dD \bv\,\leq\, 0.
\ee
Finally, the equilibrium states of the Landau operator, {\it i.e.}
functions $f$ such that $\cC(f,f) = 0$, are given by Maxwellian distribution functions:
\begin{equation}
  \label{eq:M}
\cM_{\rho, \bu,T}(\bv) \,=\, \frac{\rho}{(2\,\pi\, v_{th}^2)^{3/2}} \,e^{-\frac{|\bv-\bu|^2}{2\,v_{th}^2}},
\end{equation}
where $\rho$ is the total mass, $\bu$ the mean velocity,
\be
\label{def:rho-u}
\rho \,=\, \int_{\R} f(\bv)\,\dD \bv, \quad \rho\,\bu \,=\, \int_{\R} f(\bv)\,\bv\,\dD \bv,
\ee
and $v_{th}$ represents the thermal velocity, which depends on the
temperature $T$,
\be
\label{def:T}
3\,\rho \, T \, =\, \int_{\R} f(\bv) |\bv-\bu|^2 \,\dD \bv,
\ee
as $v_{th} = \sqrt{{k_B\,T}/{m}}$, where  $m$ represents the mass of one
particle and $k_B$ is the Boltzmann constant. Notice that these
quantities are observables and can be deduced directly from the moments
of the distribution function with respect to the velocity.

To investigate the long
time behavior of the solution, the relative entropy $\cH[f|\cM_{\rho, \bu,T}]$ plays
a major role, it is defined as
\be
\label{h:rel}
\cH[f|\cM_{\rho, \bu,T}] \,:=\,  \cH[f|] - \cH[\cM_{\rho, \bu,T}] = \int_{\R^3} f
\,\log\left(\frac{f}{\cM_{\rho, \bu,T}}\right)\,\dD \bv.
\ee
It allows to measure the distance between the solution $f$ and the
equilibrium. Indeed, thanks to the Csisz\'ar-Kullback inequality, we have
$$
\| f - \cM_{\rho, \bu,T} \|_{L^1}^2 \,\leq \, C_{CK}\, \cH[f|\cM_{\rho, \bu,T}]. 
$$

On the one hand, the collision operator (\ref{eq:1}) is classically obtained 
as a remedy of the
Boltzmann operator for a sequence of scattering cross sections which
converge in a convenient sense to a delta function at zero scattering
angle. In Coulomb collisions small
angle collisions play a more important role than collision 
resulting in large velocity changes. The original derivation of the
equation is based on this idea is due 
to Landau \cite{Landau}.  Depending on one's taste and notion of rigor several mathematical 
derivation of the equations have been performed, we mention here 
the works of Rosenbluth et al.~\cite{RMJ},   Bobylev \cite{bobylev1975}, Arsen'ev and Buryak~\cite{arsen:91}, Degond and 
Lucquin-Desreux~\cite{degond} and
Desvillettes~\cite{desvillettes}. We refer to \cite{vill:new:00}  for  a  review of the main 
mathematical aspects related to this issue.

On the other hand, Arsene'v and Peskov
\cite{arsenev0} have established the existence of weak solutions for
short time in the spacially homogeneous case for the Coulomb
potential. Later a global existence result of renormalized
solution with a defect measure has been obtained by Alexandre and
Villani \cite{alex0,villani0} in the space dependance case and for an
initial data with a finite energy.

Concerning the approximation of the Landau equation
\eqref{eq:0}-\eqref{eq:1}, various numerical methods have been studied. Depending on how to discretize the velocity space, these
methods can be categorized into two classes,
\begin{itemize}
  \item deterministic  methods based on
    finite difference schemes on a grid in velocity;
    \item  probabilistic
algorithms also known as particles methods.
\end{itemize}

In contrast with the Boltzmann equation where Monte Carlo methods 
plays a major role in numerical simulations,  the extension to long range forces, in particular 
Coulomb interactions, is more challenging and still requires a lot of attention \cite{BP, RRCD}. Most of the particle methods for Coulomb interaction have been derived more on a physical 
intuition basis and not directly from the Landau equation \cite{Nanbu,BP}.  A detailed discussion about this is beyond the aims 
of the present paper and we refer the reader to \cite{Nanbu,BP, PBM,
  RRCD}  for a more complete treatment and to \cite{PBM, DMP} for  reviews.


On the other hand, several deterministic numerical approaches have
been considered  for
the Landau collision operator  \cite{berezin,bobylev0,buet0,bcdl,
  degond0,DK,epp,lemou, lucquin,pota2,chacon1, chacon2, Zaitsev}. However, due to the computational complexity of the equation, which is essentially caused by 
the large number of variables and the three-fold collision integral, 
many papers have been devoted to simplified problems as for the isotropic case \cite{bobylev0} or
for cylindrically symmetric problems in \cite{pekker, KK}. Later from the late 1990s, a considerable amount of attention 
has been directed towards the full Landau equation.  The construction
of conservative and entropic finite difference schemes for the space
homogeneous case has been proposed by Degond and Lucquin-Desreux in
\cite{degond0} and Buet and  Cordier \cite{buet0}.  These  schemes are
built in such a way that the main physical properties are conserved at a discrete level. Positivity of the
solution and discrete entropy inequality are also satisfied. 
Unfortunately, the direct implementation of such 
schemes for space non homogeneous computations is very expensive
since the computational cost increases roughly in proportion to
the square of the number of parameters used to represent the
distribution function in the velocity space.  Thus several fast approximated algorithms to reduce the
computational complexity of these methods, based on 
multipole expansions \cite{lemou} or multigrid 
techniques \cite{buet0} have been proposed.
Although these fast schemes are able to preserve the most
relevant  physical properties, the range of applications
seems limited and the degree of accuracy of such 
approaches has not been studied.  Most of these results are provided when we neglect the
space variable in \eqref{eq:0},  we   refer to \cite{CF, DDFT}
for a direct implementation of such schemes taking into account and
space variable and self-consistent interactions in  one \cite{CF} or
two dimensions \cite{DDFT} in the
physical space and in three dimensions in velocity space.

In the meantime, a different approach based on spectral methods, has been 
proposed for the Boltzmann \cite{PR, PP} and 
Landau \cite{PRT1, PRT2, FP} collision operators.
In that case,  the spectral scheme permits 
to obtain spectrally accurate solutions with a reduction of 
the quadratic cost $N^2$ to $N\,\log_2\,N$, where $N$ is 
the total number of unknowns in the velocity space. The lack of discrete 
conservations in the spectral scheme (mass is preserved, whereas momentum 
and energy are approximated with spectral accuracy) is compensated by
its higher accuracy and efficiency. In particular the scheme allows
easily the implementation of grid-refinement techniques in the velocity space.

A detailed comparison of the spectral method with the
schemes proposed in \cite{bcdl,lemou} has been done in
\cite{Filbet}. For the same degree of freedom $N$, the computational cost of spectral
algorithm is much higher than multigrid \cite{bcdl}  or
multipole methods \cite{lemou} but this drawback can be compensated by a better
accuracy at least asymptotically when $N\gg 1$.   Let us emphasize
that more recently, spectral algorithms based on Hermite expansion of the
distribution function have been proposed \cite{LRW} allowing
to get exact conservations of mass, momentum and energy but the
computational cost is again larger than $N\,\log_2N$. This latter
method is particularly efficient when the solution remains closed to
the equilibrium since few modes may be used. 


In the present paper, we pursue the idea of spectral methods, but in
the frame of spectral collocation methods, where the distribution
function is now discretized on a grid of the velocity space instead of
computing the time evolution of Fourier modes. Fast Fourier transforms
are only applied to evaluate non-local convolution terms and discrete
derivatives. Here we will follow the idea of the spectral method
described in \cite{PRT1, PRT2} and restrict ourselves to the Coulombian case
$\gamma=-3$, which has a particular structure. We will take
advantage of this framework  to reduce the number of discrete convolutions.


The rest of the article is organized as follows. In the next section 
we describe the main features of our numerical methods and propose the
a new spectral collocation method for the approximation
of the collision operator. Next we discuss some properties of the
numerical scheme as spectral accuracy and preservation of steady states. Finally, several numerical tests for three dimensional
space homogeneous problems are presented to illustrate the efficiency
of the spectral collocation method.


\section{The numerical method}
\label{sec:nm}
\setcounter{equation}{0}
From now, we only consider the Coulombian case ($\gamma=-3$), which is
the most significant for a physical point of view. In that case, the
Landau operator has an additional property which will be the key point
of our approximation. We set $\psi(\bz)=|\bz|$ and observe that the
matrix $\bA$ in \eqref{eq:1} is such that
$$
\bA(\bz) \,=\,  \nabla^2 \psi (\bz).
$$
Hence, the Landau operator \eqref{eq:1} can be now written as
\be
\left\{
  \begin{array}{l}
    g = \psi\star f\,,
    \\[0.9em] 
\cC(f,f) = \Div_\bv\left(\nabla_\bv^2 g\,
    \nabla_\bv f \,-\, \nabla_\bv\Delta_\bv g \,f\right).
    \end{array}\right.
\label{eq:3}
\ee
This formulation will be the key point of our spectral collocation
method. Unfortunately, this formulation is not the most appropriate to
prove conservation of momentum and energy. Indeed, to prove conservation of
momentum, we may use that $g$ satisfies
$$
-\Delta^2 g \,=\, 8\,\pi \, f \quad{\rm in}\,\R^3,
$$
whereas the conservation of energy requires that we come back to the
original formulation and observe that $\bz \in \ker\bA(\bz)$. However,
this formulation is convenient  to reduce the
computational cost since  the scalar and non-local function $g$ 
may be evaluated using a discrete fast Fourier transform.  Moreover, once the function $g$ is provided, the operator $\cC(f,f)$
is a local convection/diffusion operator, hence it can be approximated
either by finite difference formula or spectral approximation.  Let us
emphasize that this formulation is widely used in physics where the potential $g=\psi\star f$ refers to the Rosenbluth potential \cite{RMJ}.

Now let us explain our spectral collocation method for \eqref{eq:3}. We  consider a set of equidistant points $(\bv_\bj)_{\bj\in \J_{n}} \subset [-R,R]^3$ with  $\J_{n}:=\llbracket
-n/2 ,n/2-1\rrbracket^3$ where $n$ is an even integer and a real
$R>0$, which determines the computational domain in velocity.   Let us
suppose that  an approximation of the distribution function is known at the mesh points $(\bv_\bj)_{\bj\in
  \J_{n}}$.

\subsection{Computation of the convolution term}
Similarly to  classical spectral methods for Boltzmann or Landau equations
\cite{PP, PRT2, FMP},  the first step will consists to reduce the
integration domain $\R^3$ to a bounded domain.  It can be shown that for a collision operator such
as  \eqref{eq:3}, we have the following property
\bpr
  \label{prop:1}
  Let $\Supp ( f ) \subset \cB(0, R)$, where $\cB(0, R)$ is the ball of radius $R$ centered in the origin. Then
  \be
  \label{eq:C}
\left\{
  \begin{array}{l}
    g \,=\, \psi^R\star f\,,
    \\[0.9em] 
\cC(f,f) = \Div_\bv\left(\nabla_\bv^2 g\,
    \nabla_\bv f \,-\, \nabla_\bv\Delta_\bv g \,f\right)\,,
  \end{array}\right.
\ee
with $\psi^R = \chi_{[-2R,2R]^3}\,\psi$, where $\chi_B$ denotes the
characteristic function in the set $B$. Moreover, we have $\bv-\bu\in \cB(0,3\,R)$. 
\epr
\begin{proof}
First we  write  $\cC(f,f)$ as
  $$
\cC(f,f)(\bv) \,=\, \Div_\bv\left(\nabla_\bv^2\int_{\R^3}\psi(\bu)
  f(\bv-\bu) \dD \bu\,
  \nabla_\bv f(\bv) \,-\, \nabla_\bv\Delta_\bv \int_{\R^3}\psi(\bu)
  f(\bv-\bu)\dD\bu\,f(\bv)\right)
  $$
  and consider  that $f$ is compactly supported in a ball $\cB(0,R)$. On the one hand,  for  $\bv$ and $\bv - \bu \in \cB(0, R)$, 
  $$
|\bu| \,=\, |\bv - \bv - \bu| \,\leq\, |\bv| \,+\, |\bv + \bu| \,\leq\,  2\,R,
$$
hence in this case, it is  enough to choose $\bu\in [-2R,2R]^3$. On the other hand, when  $\bv$ or $\bv -\bu \notin \cB(0, R)$ we get
that  $\cC(f, f)=0$ since $f\equiv 0$ in $\R^3\setminus \cB(0,R)$.

Therefore, the Landau operator may be written as
$$
 \cC(f,f) \,=\, \Div_\bv\left(\nabla_\bv^2\left[\psi^R\star f\right]\,
  \nabla_\bv f \,-\, \nabla_\bv\Delta_\bv \left[\psi^R\star
    f\right]\,f\right),
$$
which gives formula \eqref{eq:C}.

Finally $|\bv| \leq R$ and $|\bu| \leq 2R$ implies
$$
|\bv-\bu| \leq |\bv| + |\bu| \leq 3R.
$$
\end{proof}

Following Proposition \ref{prop:1}, we want to discretize the
Landau equation in a computational domain $[-R,R]^3$, with $R>0$,
hence we define a truncated collision operator $\cC^R(f,f)$ as
\be
  \label{eq:CR}
\left\{
  \begin{array}{l}
    g^R = \psi^R\star f\,,
    \\[0.9em] 
\cC^R(f,f) = \Div_\bv\left(\nabla_\bv^2 g^R\,
    \nabla_\bv f \,-\, \nabla_\bv\Delta_\bv g^R \,f\right).
  \end{array}\right.
\ee

\begin{remark}
Let us emphasize that this truncation of the domain will affect in
practice the exact conservation of momentum and energy. We refer to
\cite{bobylev0, degond} or
more recently  \cite{chacon1, chacon2} for specific numerical schemes ensuring these
conservations by modifying boundary terms. Here we favor on spectral
accuracy instead of exact conservations.
\end{remark}

As for spectral approximations to Boltzmann \cite{PP,PR}
and Landau \cite{PRT2,FP}, we aim to compute the convolution term
$g^R$ using a discrete fast  Fourier transform. To avoid aliasing
effects \cite{CHQZ},  we
need to extend the distribution function  $f$ carefully in a larger domain. Following
\cite{PP,PR, FMP} and Figure \ref{fig:0}, we choose  the cube $[-T ,
T ]^3$, where  $f\,=\, 0$ on $[-T, T ]^3 \setminus [-R,R]^3$, and extend it by periodicity to a
periodic function on $[-T , T ]$. As observed in \cite{PP,PRT2} and
illustrated in Figure \ref{fig:0}, it is enough to take $T
\geq 3\,R/2$ to prevent intersections of the regions where $f$ is different from
zero. In practice we take $T=2\,R$ and consider a set of equidistant points
$(\bv_\bj)_{\bj\in \J_{2n}} \subset [-T,T]^3$ with $T=2R$ with $\J_{2n}:=\llbracket
-n ,n-1\rrbracket^3$.

\begin{figure}
\newcommand{\ttE}{(-3,-3) rectangle (3,3)}
\newcommand{\ttF}{(3,-3) rectangle (9,3)}
\newcommand{\ttA}{(0,0)  circle (2)}
\newcommand{\ttB}{(6,0)  circle (2)}
\newcommand{\ttC}{(0,0)  circle (4)}
\newcommand{\ttD}{(6,0)  circle (4)}
\begin{center}
\begin{tikzpicture}
\draw [<->,>=stealth',line width=1pt] (9.5,0) -- (0,0) -- (0,4.5);
\draw[color=red] \ttE;
\draw[color=red] \ttF;
\fill[opacity=0.5,red]  \ttA;
\fill[opacity=0.5,red] \ttB;
\draw[black, dashed] \ttC;
\draw[black, dashed] \ttD;
\draw (0,0) node{$\bullet$};
\draw (0,0) node[below left]{$0$};
\draw (3,0) node{$\bullet$};
\draw (3,0) node[below left]{$T$};
\draw (6,0) node{$\bullet$};
\draw (6,0) node[below left]{$2T$};
\draw (9,0) node{$\bullet$};
\draw (9,0) node[below left]{$3T$};
\draw (0,-1.5) node{$\cB(0,R)$};
\draw (-0.5,-3.5) node{$\cB(0,2R)$};
\end{tikzpicture}
\end{center}
\caption{Computation of the cube $[-T,T]^3$ to avoid interactions
  between the  domain of integration $\cB(0,2R)$ and the periodized
  distribution function $f_n$ for the convolution term
  $g_{2n}^R=\psi^R\star f_{2n}$. A sufficient condition is $T>0$ is such that  $2T\geq 3\,R$.}
\label{fig:0}
\end{figure}
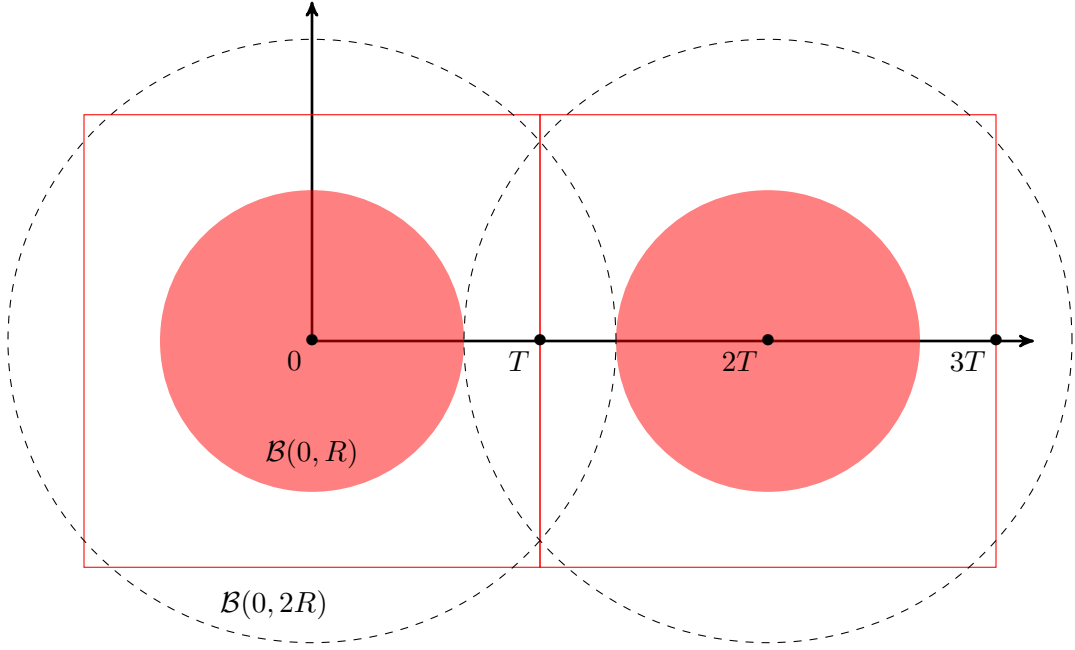

We suppose that  $f$ is only known at the mesh points $(\bv_\bj)_{\bj\in
  \J_{n}}\subset [-R,R]^3$ and set $f(\bv_\bj)=0$ for $\bj\in
\J_{2n}\setminus \J_n$.  Then we compute a discrete  Fourier transform as,
$$
\widetilde{f}_{2n}(\bk) \,:=\, \dfrac{1}{(2n)^3}\,\ds\sum_{\bj\in \J_{2n}}
f(\bv_\bj)\,e^{-\frac{i \pi}{T}\,\bk\cdot\bv_\bj}, \quad \bk\in\J_{2n}
$$
and get a function  $f_{2n}$ as a trigonometric polynomial  in the domain $[-T,T]^3$,
$$
f_{2n}(\bv)\,:=\,\ds\underset{\bk\in
  \J_{2n}}{\sum}\,\widetilde{f}_{2n}(\bk)\,e^{\frac{i\,\pi}{T}\,\bk\cdot\bv}.
$$
Due to the orthogonality relation
\begin{equation}
\frac{1}{(2n)^3} \sum_{\bj\in\J_{2n}}
e^{-\frac{i\,\pi}{T}\,\bk\cdot\bv_j} \,=\,
\left\{\begin{array}{ll}
         1, & {\rm if\,} \bk=\, 2n \,m\, {\bf 1}, \quad m=0,\,\pm 1,\,\pm 2,\ldots
         \\
         0,  &{\rm else},
       \end{array}\right.
\label{p:ortho}
     \end{equation}
     where ${\bf 1}=(1,1,1)$, we have the relation for any $\bj\in\J_{2n}$,
     $$
f_{2n}(\bv_\bj)\,=\,\ds\underset{\bk\in
  \J_{2n}}{\sum}\,\widetilde{f}(\bk)\,e^{\frac{i\,\pi}{T}\,\bk\cdot\bv_\bj}
\,=\, f(\bv_\bj).
     $$
Consequently $f_{2n}$ is the $n$-degree trigonometric interpolant of
$f$ at the nodes $(\bv_\bj)_{\bj\in \J_{2n}}$, that is,
$f_{2n}(\bv_\bj)\,=\, f(\bv_\bj)$ for all $\bj\in\J_{2n}$.
    
The first step consists in computing the convolution term
$g_{2n}^R$ by substituting the function $f_{2n}$ in the first equation of \eqref{eq:CR},  which yields
 \begin{equation}
  \label{def:gn}
g_{2n}^R \,:=\, \sum_{\bk\in \J_{2n}} \widetilde{g}_{2n}(\bk)\, e^{\frac{i\,\pi}{T}\,\bk\cdot\bv},
\end{equation}
where  $\widetilde{g}_{2n}$ is simply given by
\begin{equation}
  \label{def:hatL0}
  \widetilde{g}_{2n} (\bk) \,\,=\,\, \widetilde f _{2n} (\bk) \, \widetilde{\psi}(\bk),
\end{equation}
with
$$
  \widetilde{\psi}(\bk) \,=\,  \int_{[-T,T]^3}
  \psi^R(|\bu|)\,e^{-\frac{i\,\pi}{T}\,\bk\cdot\bu}\,\dD\bu \,=\,  \int_{[-T,T]^3}
  |\bu|\,e^{-\frac{i\,\pi}{T}\,\bk\cdot\bu}\,\dD\bu.
  $$
  This kernel  depends on the computational domain of  $f$ through the
  period $T$, hence in order to make this dependence explicit, we
  apply a simple change of variable and get 
  \begin{equation}
  \label{def:hatL1}
\widetilde{\psi}(\bk) \,=\,  \left(\frac{T}{\pi}\right)^4\,\int_{[-\pi,\pi]^3}
  |\bz|\,e^{-i\,\bk\cdot\bz}\,\dD\bz.
  \end{equation}
  Therefore, the computation of $\widetilde{\psi}(\bk) $ can be
  performed easily  by computing the Fourier coefficients corresponding to the periodic
  function $\bu \mapsto |\bu|$.

\subsection{Computation of $\cC(f_n,f_n)$}
From the trigonometric polynomials $g_{2n}^R$ defined in the cube
$[-T,T]^3$, we want to evaluate the collision operator $\cC$ at the
collocation points $(\bv_\bj)_{\bj\in\J_n}\,\subset \,[-R,R]^3$.
Therefore, we compute a new approximation $f_{n}$ as a trigonometric polynomial  in the domain $[-R,R]^3$,
$$
f_{n}(\bv)\,:=\,\ds\underset{\bk\in
  \J_{n}}{\sum}\,\widetilde{f}_{n}(\bk)\,e^{\frac{i\,\pi}{R}\,\bk\cdot\bv}\,,
$$
with
$$
\widetilde{f}_{n}(\bk) \,:=\, \dfrac{1}{n^3}\,\ds\sum_{\bj\in \J_{n}}
f(\bv_\bj)\,e^{-\frac{i \pi}{R}\,\bk\cdot\bv_\bj}, \quad \bk\in\J_{n}\,,
$$
such that  $f_{n}(\bv_\bj)\,=\, f(\bv_\bj)$ for all $\bj\in\J_{n}$.

Therefore, we define the interpolation operator $\proj_{n}$ such that
$f_{n}=\proj_{n}f$, which can be considered as the orthogonal
projection upon the space  $\cP_{n}$ of trigonometric polynomials of
degree $n$, with respect to the discrete approximation of the $L^2$
inner product,
$$
\cP_{n} \,:=\, {\rm span}\left\{ e^{\frac{i \pi}{R}\,\bk\cdot\bv},
  \quad \bk=(k_1,k_2,k_3),\, {\rm \,with\,} k_\alpha=-n/2,\ldots,n/2-1,
  \,{\rm \,for\,}\alpha\in\{1,\,2,\,3\}\right\}.
$$
Actually , the bilinear form
$$
\langle f, h \rangle  \,=\, \frac{1}{n^3} \sum_{\bj\in\J_{n}}  f(\bv_j)
\, h(\bv_j)  
$$
coincides with the inner product of $L^2([-R,R]^3)$ when $f$ and $h \in
\cP_{n}$.

The interpolant $\proj_{n}f$ of a continuous function $f$ satisfies
the identity
$$
\langle \proj_{n}f, h \rangle  \,=\, \langle f, h \rangle, \qquad h\in\cP_{n}.
$$
Thus, from $f_{n}$ and $g_{2n}$, we compute an approximation $\cC_n^R(f_{n},f_{n})$ as 
\begin{eqnarray*}
\cC_n^R(f_{n},f_{n}) &=& \Div_\bv \proj_{n}\left(\nabla_\bv^2 g_{2n}^R\,
    \nabla_\bv f_{n} \,-\, \nabla_\bv\Delta_\bv g_{2n}^R
                     \,f_{n}\right),
  \end{eqnarray*}
Now for the Fourier collocation method, we will require that the residual
$$
\cR_n = \frac{\partial f_n}{\partial t} -  \cC_n^R(f_n,f_n)
$$
vanishes at the grid points $(\bv_\bj)_{\bj\in\J_{n}}\subset [-R,R]^3$, that is,
$$
\cR_{n}(t,\bv_j) \,=\,0,\quad \bj \in\J_{n}.
$$
This yields $n^3$ equations to determine the point values
$f_n(t,\bv_j)$ of the numerical solution.  In other words, the
pseudospectral approximation $f_n$ satisfies the equation
\be
\label{eq:24}
\frac{\partial f_n}{\partial t} \,\,=\;\;  \cC_n^R(f_n,f_n).    
\ee
Notice that from $g_{2n}^R$ defined in $[-T,T]^3$, we only compute $\nabla_\bv^2 g_{2n}^R$ and
$\nabla_\bv\Delta_\bv g_{2n}^R$  at the collocation points $(\bv_\bj)_{\bj\in\J_{n}}$ in the smaller cube $[-R,R]^3$ to evaluate
$\cC_n^R(f_n,f_n)$.

\section{Properties of the pseudospectral method}
We point out that because of the periodicity  assumption on the
operator $\proj_n$, the collision operator $\cC^R_n( f_n,f_n )$ preserves in time
the mass
$$
\int_{[-R,R]^3} \cC^R_n( f_n,f_n ) \dD \bv \,=\, \ \int_{[-R,R]^3}
\Div_{\bv}\proj_n(\nabla_\bv^2 g_{2n}^R \nabla_\bv f_n -
\nabla_\bv\Delta_\bv f_n ) \dD \bv \,=\, 0.
$$
In contrast, momentum and energy are
not exactly preserved in time but as we will see below these
variations are controlled by spectral accuracy when the solution is smooth.

\subsection{Spectral accuracy}
As for spectral methods \cite{PRT2}, we can
state the following theorem.

\begin{theorem}
  \label{th:1}
  Consider $\psi^R$ given by \eqref{eq:C} and  a distribution function $f \in H^p$, with $p\geq 4$, which
  is compactly supported in the ball $\cB(0,R)$ with $R>0$. Then there
  exists a nonnegative constant $C^R>0$,  depending on  $R$, such that
  $$
\| \cC^R(f,f) - \cC^R_n(f_n,f_n) \|_{L^2}
 \,\leq\,\frac{C_R}{n^{p-4}} \, \|f\|_{H^p}\,\|f\|_{H^4}.
  $$
\end{theorem}
\begin{proof}
 First for  $f \in H^{r}$, with $r\geq 1$,  we consider $g^R=\psi^R\star f$, where $\psi^R$ is
 continuous and $\psi^R\in L^1\cap L^\infty$, hence we  get from the
Young's convolution inequality, for any $r\geq 0$
\begin{equation}
  \label{esti:0}
  \left\{
    \begin{array}{l}
\ds\|\nabla_\bv^r g^R\|_{L^\infty} \leq \|\psi^R\|_{L^2}\, \|\nabla_\bv^r
      f\|_{L^2}\,,
      \\[0.9em]
      \ds\|\nabla_\bv^r g^R\|_{L^2} \leq \|\psi^R\|_{L^1}\, \|\nabla_\bv^r
      f\|_{L^2}\,.
      \end{array}\right.
\end{equation}
Also let us  remind the following result in approximation theory : for $f\in
H^{l_0}_\#$, we have for any $0\leq l \leq l_0$,
\begin{equation}
\label{esti:sp}
\| f - \proj_n f \|_{H^l} \,\leq \, C_R \,
\frac{\|f\|_{H^{l_0}}}{n^{l_0-l}}.
\end{equation}
We now introduce
$$
\left\{
  \begin{array}{l}
\cB^R_1(f,f) =  \Div_\bv \proj_n\left(\nabla_\bv^2 g^R\,
    \nabla_\bv f \,-\, \nabla_\bv\Delta_\bv g^R
                     \,f\right),
    \\[0.9em]
   \cB^R_2(f,f) =  \Div_\bv \proj_n\left(\nabla_\bv^2 g_{2n}^R\,
    \nabla_\bv f \,-\, \nabla_\bv\Delta_\bv g_{2n}^R
    \,f\right),
    \end{array}\right.
    $$
    and estimate the numerical error as
    \begin{eqnarray}
      \label{t:0}
\| \cC^R(f,f)  - \cC_n^R(f_n,f_n) \|_{L^2} &\leq&  \| \cC^R(f,f)  -
                                                  \cB_1^R(f,f) \|_{L^2}
      \\
      &+& \| \cB^R_1(f,f)  - \cB_2^R(f,f) \|_{L^2}
\nonumber
      \\    &+& \| \cB_2^R(f,f)  - \cC_n^R(f_n,f_n) \|_{L^2}.
          \nonumber
\end{eqnarray}
The first term on the right hand side of \eqref{t:0} measures the interpolation error
of trigonometric polynomials. We set $h\,=\,\nabla_\bv^2g^R\,\nabla_\bv f - \nabla_\bv\Delta_\bv g^R
\,f$ and observe that since $f$ is compactly supported in $\cB(0,R)$,
the function $h$ is also compactly supported, hence it is periodic
and belongs tp $H^{p-3}_{\#}([-R,R]^3)$. Indeed, from the H\"older
inequality and next  the first Young convolution's inequality \eqref{esti:0}, we get
\begin{eqnarray*}
  \| h \|_{H^{p-3}}
&\leq& C\,\left( \|g^R\|_{W^{p-1,\infty}} \,
                         \| f\|_{H^{p-2}} \,+\,  \|g^R\|_{W^{p,\infty}} \,
                         \| f\|_{H^{p-3}} \right)\,,  \\
  &\leq& \, C\,\|g^R\|_{W^{p,\infty}} \,
                         \| f\|_{H^{p-2}}\,, \\
  &\leq&   C\,\|\ \psi^R\|_{L^2} \, \| f\|_{H^{p}}\,\| f\|_{H^{p-2}}\,. 
\end{eqnarray*}
Applying \eqref{esti:sp} to $h$ with $l=1$ and $l_0=p-3$, it yields
$$
\| \cC^R(f,f)  -\cB_1^R(f,f) \|_{L^2}  \leq  \|\proj_nh - h \|_{H^1}
\leq C_R \,\frac{\|h\|_{H^{p-3}}}{n^{p-4}}\,, 
$$
hence, 
\begin{equation}
\label{t:1}
\| \cC^R(f,f)  -\cB_1^R(f,f) \|_{L^2}  \leq  \frac{C_R}{n^{p-4}} \,  \|\psi^R\|_{L^2}\, \|f\|_{H^p} \, \|f\|_{H^{p-2}}\,.
\end{equation}
The second term in \eqref{esti:0} corresponds to the error between
$g^R$ and its trigonometric interpolant $g_{2n}^R$ defined in the cube
$[-T,T]^3$. Let us notice that since $g_{2n}^R \in
\cP_{2n}([-T,T]^3)$ with $T=2R$, whereas $f$ is compactly supported in
$\cB(0,R)$, as a consequence the function $\nabla_\bv^2 g_{2n}^R \,\nabla_\bv f -
\nabla_\bv\Delta_\bv g_{2n}^R \, f$ is smooth and compactly supported in
$\cB(0,R)$ and its restriction to the cube  $[-R,R]^3$ can be regarded
as a smooth and periodic function. Again applying the
stability estimate to any $u\in H^1_\#$
$$
\|\Div_\bv \left(\proj_n u\right)\|_{L^2}\,\leq \,\|\proj_n u\|_{H^1}\,\leq\, C\, \|u\|_{H^1}
$$
and the H\"older inequality, we get 
\begin{eqnarray*}
\| \cB_1^R(f,f)   -\cB_2^R(f,f) \|_{L^2}  & \leq & \|\nabla_\bv f \,
\nabla_\bv^2 (g^R - g^R_{2n}\|_{H^1} \,+\, \| f \, \nabla_\bv\Delta_\bv (g^R - g_{2n}^R)\|_{H^1}\,,
  \\
                                          & \leq & C\,\left(
                                                   \|\nabla_\bv^2
                                                   f\|_{L^\infty}
                                                   \,+\, \|f\|_{L^\infty} \right) \, \| g^R -  g^R_{2n} \|_{H^4}.
\end{eqnarray*}
 Thus, using that $f$ and $\nabla_\bv^2 f \in H^2 \subset L^\infty$
 with continuous embedding, then  applying \eqref{esti:sp}  to $g^R$ with $l=4$ and $l_0=p$ and next \eqref{esti:0}, it
 gives the following estimate on the second term of the right hand
 side in \eqref{t:0},
\begin{equation}
\label{t:2}
  \| \cB_1^R(f,f)   -\cB_2^R(f,f) \|_{L^2}  \,\leq\,\frac{C_R}{n^{p-4}} \,  \|\psi^R\|_{L^1}\, \|f\|_{H^p}\, \|f\|_{H^4}.
\end{equation}
Finally we estimate the third term in \eqref{t:0} which takes into account
the error between $f$ and its interpolant $f_n$ in the cube
$[-R,R]^3$. We have already seen that $\cB_2^R(f,f)$ is compactly
supported and it can be considered
as a periodic smooth function in $[-R,R]^3$ whereas $\cC_n^R$ is 
the divergence of the interpolant of the  function $h_n\,=\,\proj_n\left(\nabla_\bv^2
g^R_{2n} \nabla_\bv f_n - \nabla_{\bv}\Delta_\bv g_{2n}^R\,
f_n\right)$, which is also periodic in $[-R,R]^3$ and infinitely smooth. 
\begin{eqnarray*}
\| \cB_2^R(f,f) - \cC^R_n(f_n,f_n)  \|_{L^2}  &\leq& \left( \|\nabla_\bv^2
g_{2n}^R\|_{L^\infty}+ \|\nabla_\bv^3
                                                     g_{2n}^R\|_{L^\infty}\right) \, \| \nabla_\bv( f- f_n)\|_{H^1}
  \\
  &+&
\left( \|\nabla_\bv^3 g_{2n}^R\|_{L^\infty}+\|\nabla_\bv^4 g_{2n}^R\|_{L^\infty}\right) \, \|  f- f_n\|_{H^1},
\end{eqnarray*}
which gives applying the same argument as before,
\begin{equation}
\label{t:3}
  \| \cB_2^R(f,f) - \cC^R_n(f_n,f_n) \|_{L^2}  \,\leq\,\frac{C_R}{n^{p-2}} \,  \|f\|_{H^p}\, \|\psi^R\|_{L^2}\, \|f\|_{H^4}.
\end{equation}
Gathering \eqref{t:1}-\eqref{t:3} into \eqref{t:0}, we get the spectral
accuracy, there exists a constant $C_R>0$ such that,
$$
 \| \cC^R(f,f) - \cC^R_n(f_n,f_n) \|_{L^2}
 \,\leq\,\frac{C_R}{n^{p-4}} \, \|f\|_{H^p}\,\|f\|_{H^4}.
$$

\end{proof}

\begin{remark}
  Let us emphasize that in \cite[Corollary 2.1]{PRT1, PRT2}, the
  consistency error is slightly  better since with the
  same assumptions on $f\in H^p$, the authors get
  $$
\| \cC^R(f,f) - \cC^R(f_n,f_n) \|_{L^2}
 \,\leq\,\frac{C_R}{n^{p-2}} \, \left(\|f\|_{H^p}\,+\, \|\cC^R(f_n,f_n)\|_{H^p}\right),
 $$
 where $\cC^R(f_n,f_n)$ represents now a spectral approximation. In
 our case, we provide an estimate of the last term in the right-hand
 side for the Coulombian case, hence  the error is deteriorated.
  \end{remark}
No information is available on the discrete equilibrium states, 
the decay of the numerical entropy and the preservation of positivity.

\subsection{Comparison with the classical spectral method \cite{PRT2}}

We have seen in the previous section that the spectral collocation
method provides an approximation of the collision operator with
spectral accuracy and preserves mass as the classical spectral method does
\cite{PRT2}. However, in term of computational efficiency the spectral
collocation method  has some advantages.

On the one hand,  for a distribution function $f$ such that $\Supp f
\subset \cB(0,R)$, the computational domain is kept relatively small
$[-R,R]^3$,   since the distribution function is simply extended to zero for the
  evaluation of the convolution term $g=\psi\star f$, whereas the
  computational domain of the spectral method is taken as $[-T,T]^3$
  with $T\simeq 2R$.

On the other hand,  as it is pointed out in \cite[Section 3]{PRT2}, the
  evaluation of the Fourier kernel of $\cC(f,f)$ requires nine
  convolutions of the form
  $$
  \widehat{q}(\bk) \,=\, \sum_{\bll+\bmm=\bk} \alpha(\bll)\,\widehat{f}(\bll)\,  \beta(\bmm)\,\widehat{f}(\bmm)\,, 
  $$
where $\widehat{f}(\bll)$ is the Fourier coefficient of $f$. Each
convolution term  requires three discrete Fourier transforms  with a
computational cost of $O\left((2n)^3\,\log((2n^3))\right)$ to avoid
aliasing \cite{CHQZ}, where $n^3$ represents the total number of
Fourier coefficients. This computational cost is mainly due to the
fact that $\cC(f,f)$ is quadratic with respect to $f$, hence 
nonlinear terms become  discrete convolutions in Fourier variable,
whereas the non-local term becomes local in Fourier variable.

Here we take advantage of the Coulombian case and the fact we
discretize the nonlinear operator in the velocity space. Indeed, first the matrix
  \eqref{def:A} can be written as $\nabla^2_\bv \psi$, hence we only
  compute one convolution term $g=\psi\star f$, which corresponds to a
  nonlocal term in velocity then we differentiate
  its trigonometric  polynomial approximation.  The cost of this
  convolution is reduced since $\psi$ is fixed and its discrete
  Fourier coefficient might be stored.   Finally, since we compute
  $\cC(f,f)$ at collocation points $(\bv_\bj)_{\bj\in\J_n}\subset
  [-R,R]^3$, the nonlinear terms do not require an additional cost,
  whereas the non-local term is a convolution, which can be evaluated
  with a cost $O((2n)^3\log((2n)^3))$ to compute $f_{2n}$.

  Finally the last comment concerns the Fourier coefficients. The
  classical spectral method may be written in the form \cite[Section 3]{PRT2},
  $$
\frac{\partial \widehat{f}(\bk)}{\partial t} \,=\, \sum_{\bmm\in
\J_{n}} \widehat{\beta}(\bk-\bmm,\bmm)  \widehat{f}(\bmm)\, \widehat{f}(\bk-\bmm)\,,
$$
with
$$
\widehat{\beta}(\bll,\bmm) = |\bll |^2 \, {\rm
  Tr}\left(\cI(\bmm)\right) \,-\, \bll^t \,\cI(\bmm)\,\bll\,, 
$$
and
$$
\cI(\bmm) \,=\, \int_{\cB(0,\pi)} \frac{1}{|\bz|} \,\left( \bI - \frac{\bz\otimes \bz}{|\bz|^2}\right)\, e^{-i\bmm\cdot \bz}\dD \bz\,. 
$$
This latter integral can be reduced to a one dimensional integral but
the integrand has a singularity in $\bz=0$, hence it has to be
computed carefully using a  recursive quadrature formula. With the
spectral collocation, we only need to evaluate \eqref{def:hatL1},
given by
$$
\widetilde{\psi}(\bk) \,=\,  \left(\frac{T}{\pi}\right)^4\,\int_{[-\pi,\pi]^3}
  |\bz|\,e^{-i\,\bk\cdot\bz}\,\dD\bz,
  $$
  which can be done quickly using a fast Fourier transform to $\bu
  \mapsto |\bu|$. This latter function is not infinitively smooth in the
  cube $[-\pi,\pi]^3$, hence a large number of collocation points  has
  to be used to evaluate $\widetilde{\psi}(\bk)$ and to preserve
  spectral accuracy. Let us notice that in the proof of Theorem \ref{th:1}, we did not
  tale into account the error due to the approximation of these
  Fourier coefficients.

  \subsection{Steady-state-preserving method}
  A major drawback of \eqref{eq:24} is the lack of exact conservations and,
  as a consequence, the incapacity of the scheme to preserve
  Maxwellian distribution function. In this section, we overcome this
  drawback thanks to a new reformulation of the method already propose
  in the context of Boltzmann equation \cite{FPR} that permits to
  preserve the spectral accuracy and to capture the long-time behavior
  of the system.

Let us denote  $\cM_n$  a trigonometric polynomial belonging to
$\cP_n$ and
interpolating the Maxwellian distribution $\cM$, given in \eqref{eq:M}, at the collocation
points $(\bv_\bj)_{\bj\in\J_n}$. Note that due to space homogeneity,
$\cM$ does not change in time and so does $\cM_n$. We simply modify the
previous method \eqref{eq:24} into
\be
\label{eq:24-bis}
\frac{\partial f_n}{\partial t} \,\,=\;\; \cL^R_n(f_n,f_n)\,,
\ee
where
$$
\cL^R_n(f_n,f_n)\,:=\,  \cC_n^R(f_n,f_n) - \cC_n^R(\cM_n,\cM_n)\,.    
$$
Observe that thanks to Theorem \ref{th:1}, this additional term does
not affect the spectral accuracy for smooth solutions to the Landau
equation \eqref{eq:0}-\eqref{eq:1}. Indeed, we have
$$
\| \cC^R_n(\cM_n,\cM_n) \|_{L^2}\,=\,\| \cC^R_n(\cM_n,\cM_n)-\cC^R(\cM,\cM) \|_{L^2}
 \,\leq\,\frac{C_R}{n^{p-4}} \, \|\cM\|_{H^p}\,\|\cM\|_{H^4}.
 $$
 Hence we get the same result as Theorem \ref{th:1} to
 \eqref{eq:24-bis} in term of accuracy but also the conservation of the
 consistent steady state $\cM_n$. We will illustrate in the next
 section the advantage of the steady state preserving method for the
 long time behavior of the solution to \eqref{eq:0}-\eqref{eq:1}.
 
  \section{Numerical simulations}

In this section we perform some numerical tests of the scheme
\eqref{eq:24} and \eqref{eq:24-bis}, to
check the accuracy and the efficiency of the method.

All calculations have been performed by a third-order Runge-Kutta
scheme, with fixed time step.  Of course, the Landau equation suffers
from the stiffness typical of diffusion equations. The stability
condition requires that the time step scales with the square of the
velocity step. This means that by doubling the number of Fourier modes
per direction, the total number of time steps becomes four times
bigger to compute up to the same final time. We have not performed a
stability analysis of the scheme, and the stability condition used in
the computation has been found empirically. No attempt has been made
to overcome the numerical stiffness of the problem caused by
diffusion. Although this is a very important issue and deserves a
careful study, it is beyond the scope of the present paper and we
refer to \cite{LM05} or \cite{FJ, BFR} for research
directions on this topic. In our numerical test, we first look for the
time step $\Delta t$ when $n$ is small and then choose $\Delta
t=O(1/n^2)$ when $n$ is increased.

\subsection{Order of accuracy}
Unfortunately for Coulombian interactions, there is no explicit
solution to the evolution problem \eqref{eq:0}  as it holds true when
$\gamma=0$ in \eqref{eq:1}-\eqref{def:A}. However, since the
Maxwellian distribution \eqref{eq:M} is a stationary state for
\eqref{eq:1}-\eqref{def:A} and since the spectral collocation method
\eqref{eq:24} does not preserve  steady states, the Maxwellian may be
used to evaluate the accuracy. To this aim we choose the initial data
$$
f_0 = \cM_{1, 0,1}  \quad {\rm in }\, [-7,7]^3, 
$$
with a small time step $\Delta t=0.005$. Then, we perform numerical
simulations on the time intervall $t \in [0,1]$ with respectively
$n^3=8^3$, $n^3=16^3$ and $n^3=32^3$. In Figure \ref{fig:t1-1}, we
represent the time evolution of the the relative entropy
\eqref{h:rel} and also  $L^1$, $L^2$ and $L^\infty$
error norms. The rate at which the error decays with the increase of the
number of modes is an indication of spectral accuracy. Let us
emphasize that as we mentioned before, the crucial point is to get an
accurate approximation of Fourier coefficients $(\widetilde{\psi}(\bk)
)_\bk$ in order to get spectral accuracy. Indeed, for this test we use
$1024$ collocation points in each direction to evaluate
\eqref{def:hatL1}.

\begin{figure}[ht!]
\begin{center}
  \begin{tabular}{cc}
    \includegraphics[width=8cm]{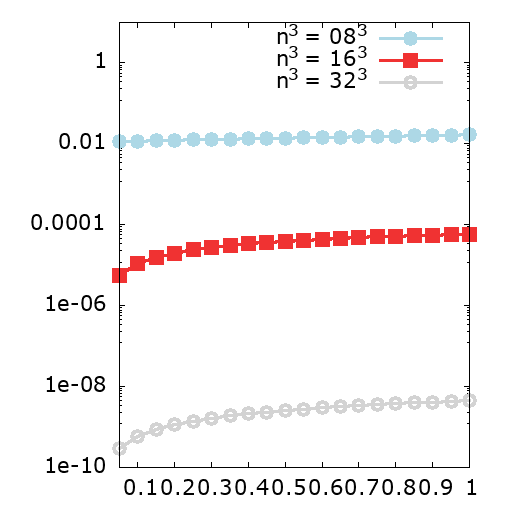} &
  \includegraphics[width=8cm]{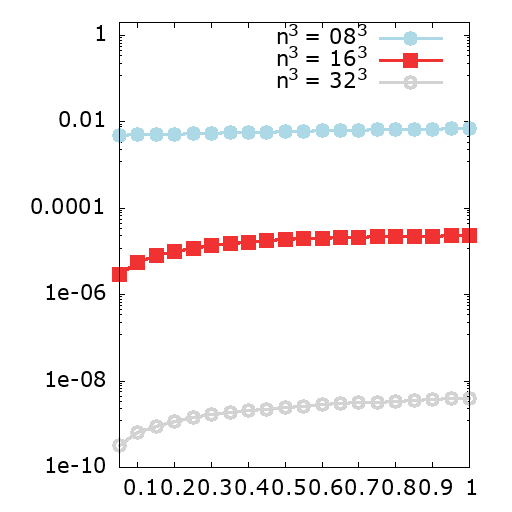}
      \\
 (a) $\|f_N-\cM_{1,0,1}\|_{L^2}$ & (b) $\|f_N-\cM_{1,0,1}\|_{L^\infty}$  
\end{tabular}
  \end{center}
\caption{{\bf Order of accuracy:} time evolution of  the $L^2$ and $L^\infty$ error norms for
  the scheme \eqref{eq:24}.}
\label{fig:t1-1}
\end{figure}

\begin{table}[!ht] 
	\centering
	\vspace{0.15in}
	\begin{tabular}{|l|l|l|l|l|}
          \hline
           $n$ & $L^2$ error norm & Order &
                                                                    $L^\infty$ error norm & Order \\
          \hline
          08 &  1.66\,$\times \,10^{-2}$ & --   &  7.03\, $\times \,10^{-3}$  & --
          \\ \hline
          16 & 5.73\,$\times \,10^{-5}$ & 8.17 & 2.30\,$\times \,10^{-5}$ & 8.24\\
          \hline
          32 & 4.46\,$\times \,10^{-9}$ & 13.64 & 3.99\,$\times \,10^{-9}$ & 12.9  \\
          \hline
        \end{tabular}
	\caption{{\bf Order of accuracy:} numerical error for $L^2$
          and $L^\infty$ norms and order of accuracy.}
	\label{tab:1}
      \end{table}
        
\subsection{Rosenbluth problem}
We choose the initial condition as
$$
f_0(\bv) \,=\,
\frac{1}{S^2} \exp\left( -S\,\frac{(|\bv| - \sigma)^2}{\sigma^2}\right),
$$
 with $\sigma = 0.3$ and $S = 10$ and the integration time is $T_0
 =50$ with $\Delta t = 0.1$ in the computational domain $[-1,1]^3$.

 This test is used to compute the time evolution of the numerical solution \eqref{eq:24-bis} and to compare the
 results with those obtained in \cite{RMJ, PRT2}.

 First, we again illustrate the spectral accuracy by plotting the time
 evolution of total momentum and temperature which are not exactly
 preserved by the spectral collocation method. In  Figure
 \ref{fig:t2-1}, we present the time evolution of momentum and
 temperature in log scale and observe that their variations become
 smaller and smaller when the number of collocation point is
 increasing. With only $n=24$, the variations of momentum and
 temperature are of order $10^{-5}$.
\begin{figure}[ht!]
\begin{center}
  \begin{tabular}{cc}
  \includegraphics[width=8cm]{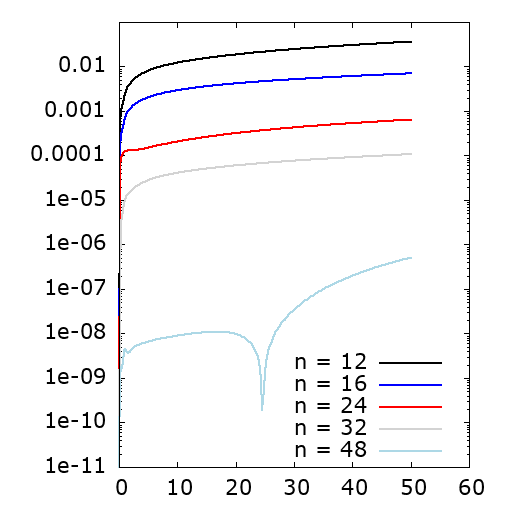} &
  \includegraphics[width=8cm]{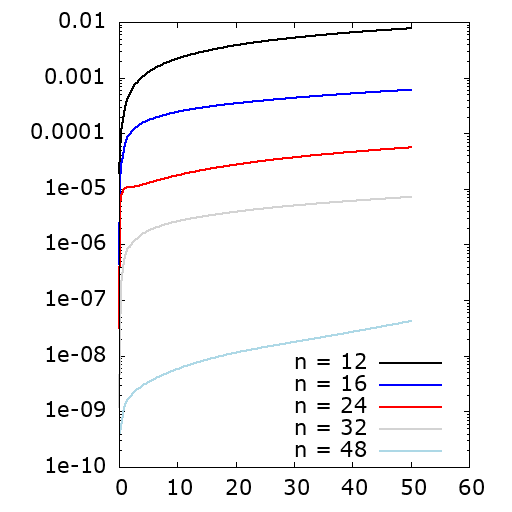}
      \\
 (a) $\|\rho\,\bu(t)\|$ & (b) $|T(t)-T(0)|$  
\end{tabular}
  \end{center}
\caption{{\bf Rosenbluth problem:} time evolution of the variations of
  momentum \eqref{def:rho-u} and
  temperature \eqref{def:T}}
\label{fig:t2-0}
\end{figure}

We compute the time evolution of the distribution function and present
the numerical results in  Figures \ref{fig:t2-1} and \ref{fig:t2-2} where
computations were performed respectively with $n^3 = 16^3$, $n^3 = 24^3$ and $n^3=
32^3$  collocation points.

On the one hand, we present in Figure \ref{fig:t2-1}, the time evolution of the relative entropy
\eqref{h:rel} and the fourth order moment of the distribution function
\be
\label{eq:M4}
M_4(t) = \int_{\R^3} f(t,\bv) \,\dD \bv.
\ee
The relative entropy $\cH[f|\cM_{\rho, \bu,T}]$ and the fourth order moment $M_4$ are computed by
discretizing the expressions \eqref{h:rel} and \eqref{eq:M4} on the velocity grid by a
straightforward formula.

We compare these quantities  with a reference solution obtained with
$n^3=128^3$ collocation points. We observe that with only $n=24$ in
each direction, we get an accurate solution.  

\begin{figure}[ht!]
\begin{center}
  \begin{tabular}{cc}
  \includegraphics[width=8cm]{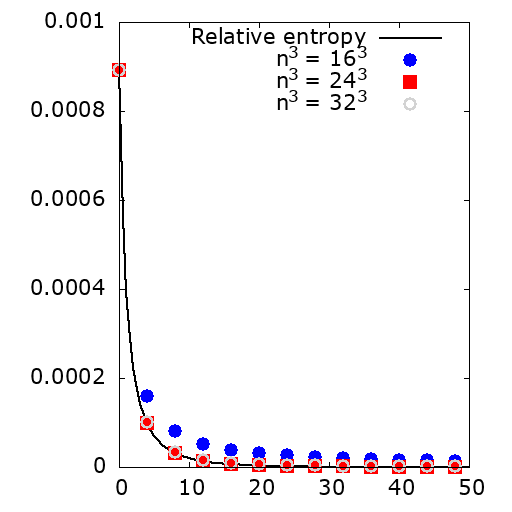} &
  \includegraphics[width=8cm]{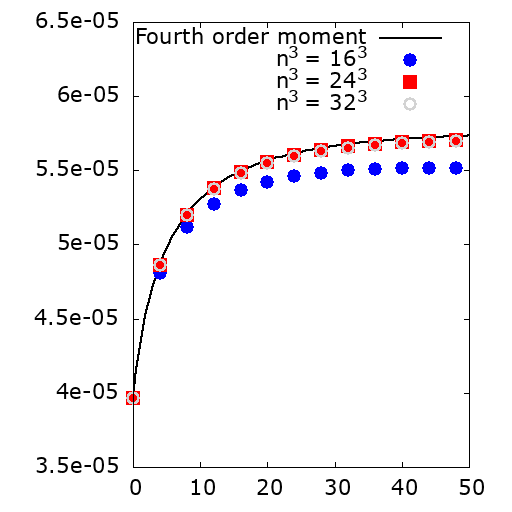}
      \\
 (a) $\cH[f|\cM_{\rho, \bu,T}]$ & (b) $M_4[f]$  
\end{tabular}
  \end{center}
\caption{{\bf Rosenbluth problem:} time evolution of the relative
  entropy \eqref{h:rel} and the fourth order moment \eqref{eq:M4}.}
\label{fig:t2-1}
\end{figure}

On the other hand, in Figure \ref{fig:t2-2} , we show the cross section of the distribution function at times $t = 0, 0.5, 1.5, 3, 5$, and
$50$. The results are in good agreement with those presented in
\cite{RMJ, PRT2} even if the time scale does not corresponds since we
take all physical constants equal to one.

\begin{figure}[ht!]
\begin{center}
  \begin{tabular}{cc}
  \includegraphics[width=8.cm]{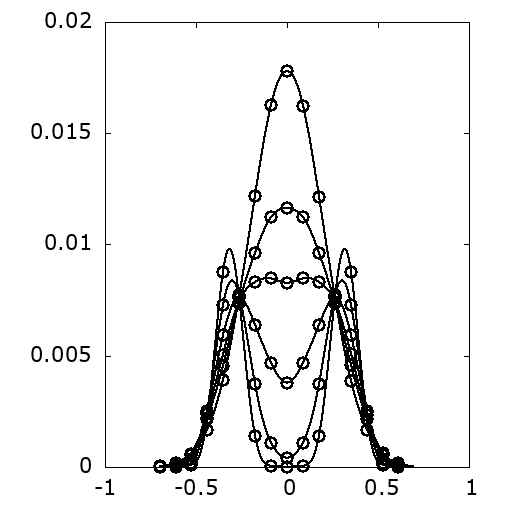} &
  \includegraphics[width=8.cm]{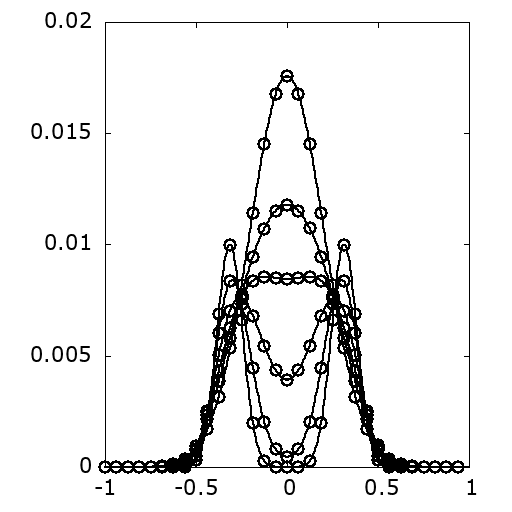}
      \\
 (a) $n^3=16^3$ & (b) $n^3=32^3$  
\end{tabular}
  \end{center}
\caption{{\bf Rosenbluth problem:} time evolution the distribution
  function $f(t,0,0,v_z)$ for (a) $n^3=16^3$ and (b)  $n^3=32^3$ at
  time $t = 0, 0.5, 1.5, 3, 5$, and $50$.}
\label{fig:t2-2}
\end{figure}

\subsection{Trend to equilibrium for the sum of two Gaussians}
We now choose the initial condition as
$$
f_0(\bv)\,=\,\frac{1}{2(2\pi\sigma^2)^{3/2}}\,\left[
  \exp\left(-\frac{|\bv - 2 \sigma\ber_1|^2}{2\sigma^2}\right) \,+\ \exp\left(-\frac{|\bv - 2 \sigma \ber_1|^2}{2\sigma^2}\right)\right],
$$
with $\sigma = \pi/10$ and $\ber_1=(1,0,0)$. We choose $\Delta t=5.\,
10^{-3}$ when $n^3=32^3$ and the computational domain is $[-R,R]^3$
with $R=2.75$.

This test is used to
compute the evolution of the entropy and the pressure tensor defined by
\be
\label{eq:P}
\bP(t) \,=\, \int_{\R^3} (\bv-\bu)\otimes (\bv-\bu) \,f(t,\bv)\,\dD \bv\,,
\ee
where $\bu$ is  the mean velocity.

In this test, we compare the long time behavior of two numerical
solutions given by \eqref{eq:24} and by the steady state preserving
method \eqref{eq:24-bis} with the same numerical resolution. 

In the Figures \ref{fig:t3-1} and \ref{fig:t3-2}, the dotted lines
represent the numerical solutions obtained with $32^3$ modes whereas  the
continuous line represent a reference solution using  $64^3$
modes. Both schemes give similar results when we observe the time
evolution of the relative entropy $\cH[f|\cM_{\rho, \bu,T}]$, which decreases to
zero, when $t$ becomes large.  However, passing to log
scale allows to observe a difference on the two solutions around the
equilibrium since the relative entropy corresponding to the solution
\eqref{eq:24} seems to saturate for $t \geq 1.5$. 

Furthermore, in  Figure \ref{fig:t3-2}, we present the time evolution of the
pressure tensor $\bP$ and the fourth order moment. Again, the steady
state preserving method gives satisfying results when time become
large. Finally, in Figure \ref{fig:t3-3}, we propose several snapshots
of the projection of $f$ at different time,
$$
F(t,v_x,v_y) = \int_{\R} f(t,\bv)\,\dD v_z.
$$

\begin{figure}[ht!]
\begin{center}
  \begin{tabular}{cc}
  \includegraphics[width=8cm]{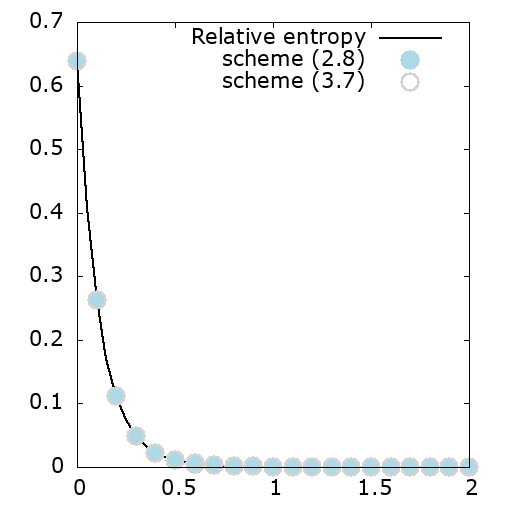} &
  \includegraphics[width=8cm]{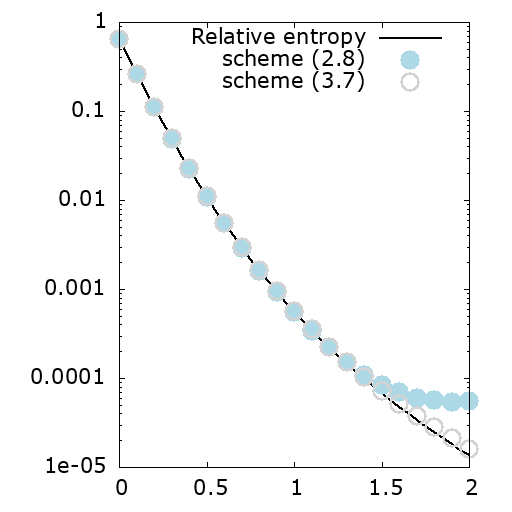}
      \\
 (a) $\cH[f|\cM_{\rho, \bu,T}]$ & (b) $\cH[f|\cM_{\rho, \bu,T}]$ in log scale
\end{tabular}
  \end{center}
\caption{{\bf Trend to equilibrium for the sum of two Gaussians: } time evolution of the relative
  entropy \eqref{h:rel}.}
\label{fig:t3-1}
\end{figure}

\begin{figure}[ht!]
\begin{center}
  \begin{tabular}{cc}
  \includegraphics[width=8cm]{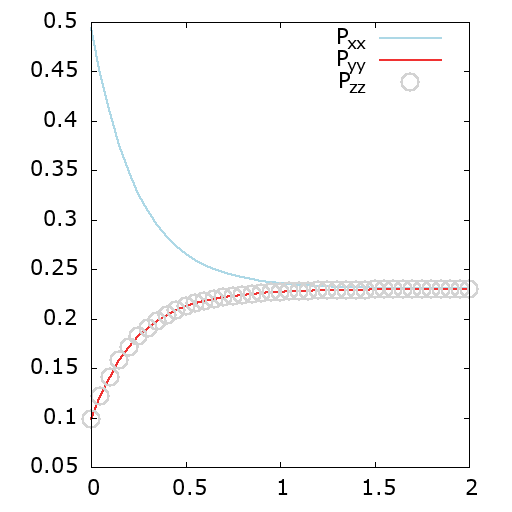} &
  \includegraphics[width=8cm]{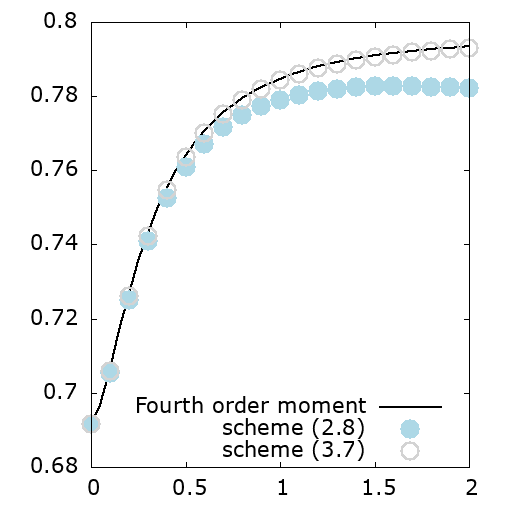}
      \\
 (a) $\bP$ & (b) $M_4$ 
\end{tabular}
  \end{center}
\caption{{\bf Trend to equilibrium for the sum of two Gaussians: } time
  evolution of the Pressure tensor $\bP$ in \eqref{eq:P} and the fourth
  order moment $M_4$ given by \eqref{eq:M4}.}
\label{fig:t3-2}
\end{figure}

\begin{figure}[ht!]
\begin{center}
  \begin{tabular}{cc}
  \includegraphics[width=8.cm]{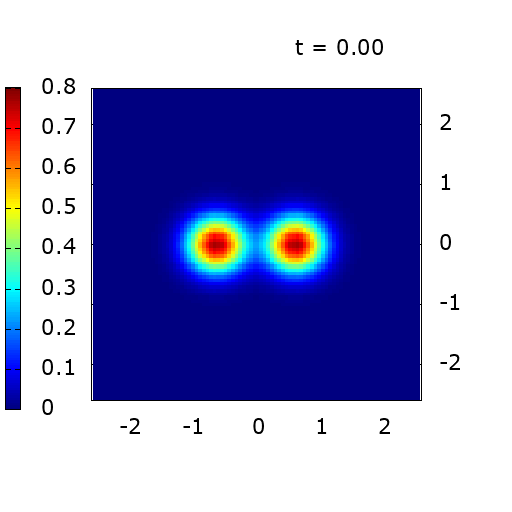} &
  \includegraphics[width=8.cm]{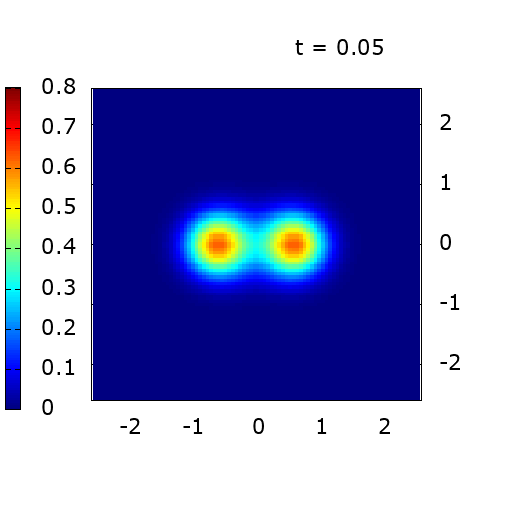}
      \\
  \includegraphics[width=8.cm]{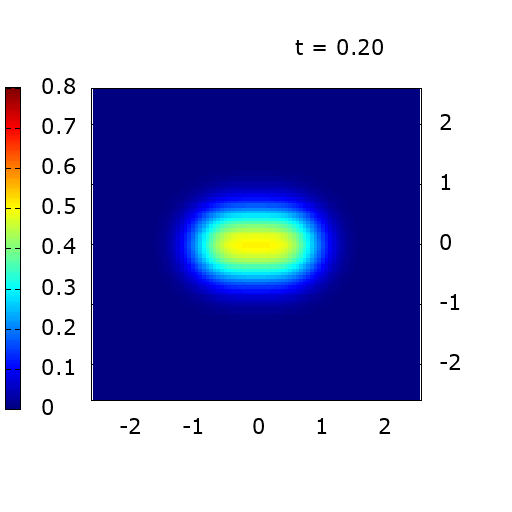} &
  \includegraphics[width=8.cm]{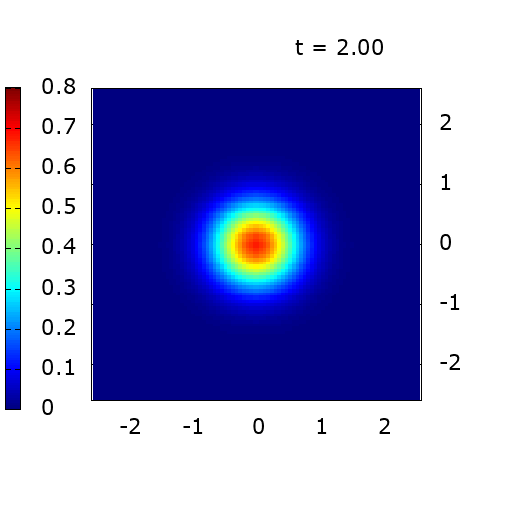}
\end{tabular}
  \end{center}
\caption{{\bf Trend to equilibrium for the sum of two Gaussians: } time
  evolution the projection of the distribution
  function $f$ on the plane $(O, \,v_x, \,v_y)$ at different time.}
\label{fig:t3-3}
\end{figure}

\section{Conclusion.}
We have presented an efficient and accurate numerical method 
to solve the space homogeneous Landau equation for plasma physics. The collision operator is solved in only $O(n^3\log_2 n)$ operations 
using a spectral collocation method and discrete fast Fourier transforms to evaluate the derivatives and the non-local term. We take
advantage of the particular structure of Coulombian interactions to reduce the number of discrete convolutions.  This approach highly improves the efficiency of the method and seems 
to be a good compromise between accuracy and computational cost. The scheme conserves mass, and approximate momentum and energy with
spectral accuracy provided that a sufficiently large support is the
velocity space is used.

This is a first step in the construction of an effective scheme for the
numerical solution of the Landau equation. We have seen that our method
is suitable for treating cases when the distribution function can be
effectively described with a reasonably low number of collocation
points, in particular, when the distribution
function is  smooth. For non-smooth solutions, this require more
investigations. In the near future we plan to extend this approach to
spatially nonhomogeneous situations following the previous works in 
\cite{CF,DDFT} and also to find  a suitable time discretization to
avoid the restriction (parabolic CFL)  on the time step \cite{FJ,LM05}.

\section*{Acknowledgement}

The author is supported by the EUROfusion Consortium and has
received funding from the Euratom research and training programme
2014-2018 under grant agreement No 633053. The views and opinions
expressed herein do not necessarily reflect those of the European Commission.


\end{document}